\documentclass[a4paper]{article}

          \usepackage[paperwidth=7in, paperheight=10in, margin=.875in]{geometry}

          \usepackage{graphicx}
          \usepackage[francais,english]{babel}
          \usepackage[utf8]{inputenc} 
          \usepackage{amsmath} 
          \usepackage{theorem} 
          \usepackage{mathrsfs}
          \usepackage{amssymb}
          \usepackage{amsfonts}
          \usepackage{amsmath}
          \usepackage{color}

          \def\RR{\mathbb{R}}
          \def\pa{\partial}

          \newtheorem{thm}{Theorem}[section]
          
          \newtheorem{lem}[thm]{Lemma}

\newenvironment{proof}{\noindent{\bf Proof.~}}
{{\mbox{}\hfill {\small \fbox{}}\\}}

\begin{document}

\title{Incompressible limit of a mechanical model for tissue growth with non-overlapping constraint.
\footnote{S.H. acknowledges support from the Imperial/Crick PhD program. N.V. acknowledges partial support from the ANR blanche project {\it Kibord} No ANR-13-BS01-0004 funded by the French Ministry of Research.
Part of this work has been done while N.V. was a CNRS fellow at Imperial College, he is really grateful to the CNRS and to Imperial College for the opportunity of this visit.
The authors would like to express their sincere gratitude to Pierre Degond for his help and its suggestions during this work.}
}

\author{Sophie Hecht\thanks{ Francis Crick Institute,
        1 Midland Rd, Kings Cross, London NW1 1AT, UK - 
        Imperial College London, 
        South Kensington Campus
London SW7 2AZ, UK email(sophie.hecht15@imperial.ac.uk)}
\and{Nicolas Vauchelet\thanks{ LAGA - UMR 7539,
Institut Galilée,
Université Paris 13,
99 avenue Jean-Baptiste Clément,
93430 Villetaneuse - France,  email(vauchelet@math.univ-paris13.fr) }}}

\date{}

         \maketitle
\thispagestyle{empty}

          \begin{abstract}
A mathematical model for tissue growth is considered. This model describes the dynamics of the density of cells due to pressure forces and proliferation. It is known that such cell population model converges at the incompressible limit towards a Hele-Shaw type free boundary problem. The novelty of this work is to impose a non-overlapping constraint. This constraint is important to be satisfied in many applications. One way to guarantee this non-overlapping constraint is to choose a singular pressure law. The aim of this paper is to prove that, although the pressure law has a singularity, the incompressible limit leads to the same Hele-Shaw free boundary problem.
          \end{abstract}



\section{Introduction}

Mathematical models are now commonly used in the study of growth of cell tissue.
For instance, a wide literature is now available on the study of the tumor growth through mathematical modeling and numerical simulations \cite{Bellomo1,Bellomo2,Friedman,Lowengrub}. 
In such models, we may distinguish two kinds of description: Either they describe the dynamics of cell population density (see e.g. \cite{Byrne,BenAmar}), or they consider the geometric motion of the tissue through a free boundary problem of Hele-Shaw type (see e.g. \cite{Greenspan,FriedmanHu,Cui,Lowengrub}).
Recently the link between both descriptions has been investigated from a mathematical point of view thanks to an incompressible limit \cite{PQV}.

In this paper, we depart from the simplest cell population model as proposed in \cite{BD}.
In this model the dynamics of the cell density is driven by pressure forces and cell multiplication.
More precisely, let us denote by $n(t,x)$ the cell density depending on time $t\geq 0$ and position $x\in \RR^d$, and by $p$ the mechanical pressure. The mechanical pressure depends only on the cell density and is given by a state law $p=\Pi(n)$.
Cell proliferation is modelled by a pressure-limited growth function denoted $G$.
Mechanical pressure generates cells displacement with a velocity whose field $v$ is computed thanks to the Darcy's law.
After normalizing all coefficients, the model reads
\begin{align*}
& \pa_t n + \nabla\cdot (n v) = n G(p),  \quad \mbox{ on } \RR^+\times\RR^d,  \\
& v = - \nabla p, \qquad p = \Pi(n).
\end{align*}
The choice $\Pi(n)= \frac{\gamma}{\gamma -1}n^{\gamma-1}$ has been taken in \cite{PQV,PQTV,PV}.
This choice allows to recover the well-known porous medium equation for which a lot of nice mathematical properties are now well-established (see e.g. \cite{Vazquez}). The incompressible limit is then obtained by letting $\gamma$ going to $+\infty$.

However, this state law does not prevent cells to overlap. In fact, it is not possible with this choice to avoid the cell density to take value above $1$ (which corresponds here to the maximal packing density after normalization).
A convenient way to avoid cells overlapping is to consider a pressure law which becomes singular when the cell density approaches $1$. 
Such type of singularity is encountered, for instance, in the kinetic theory of dense gases where the interaction between molecules is strongly repulsive at very short distance \cite{Chapman}.
Similar singular pressure laws have been also considered in \cite{Degond,Degond2} to model collective motion, in \cite{Berthelin,Berthelin2} to model the traffic flow, and in \cite{Ewelina} to model crowd motion (see also the review article \cite{Maury}). 
Then, in order to fit this non-overlapping constraint, we consider the following simple model of pressure law given by
$$
P(n)=\epsilon \frac{n}{1-n}.
$$

Finally, the model under study in this paper reads, for $\epsilon>0$,
\begin{align}
\partial_t n_\epsilon - \nabla \cdot (n_\epsilon \nabla p_\epsilon) = n_\epsilon G(p_\epsilon),  \label{eq:n} \\
p_\epsilon = P(n_\epsilon) = \epsilon \frac{n_\epsilon}{1-n_\epsilon}.  \label{eq:p}
\end{align}
This system is complemented by an initial data denoted $n_\epsilon^{ini}$.
The aim of this paper is to investigate the incompressible limit of this model,
which consists in letting $\epsilon$ going to $0$ in the latter system.

At this stage, it is of great importance to observe that from \eqref{eq:n}, we may deduce an equation for the pressure by simply multiplying \eqref{eq:n} by $P'(n_\epsilon)$ and using the relation $n_\epsilon=\frac{p_\epsilon}{\epsilon+p_\epsilon}$ from \eqref{eq:p},
\begin{equation}\label{eq:p1}
\partial_t p_{\epsilon} - (\frac{p_{\epsilon}^2}{\epsilon}+p_{\epsilon})\Delta p_{\epsilon} -  |\nabla p_{\epsilon}|^2 = (\frac{p_{\epsilon}^2}{\epsilon}+p_{\epsilon}) G(p_{\epsilon}).
\end{equation}
Formally, we deduce from \eqref{eq:p1} that when $\epsilon\to 0$, we expect to have the relation
\begin{equation}\label{eq:p0}
-p_0^2 \Delta p_0 = p_0^2 G(p_0).
\end{equation}
Moreover, passing formally to the limit into \eqref{eq:p}, it appears clearly that $(1-n_0) p_0=0$.
We deduce from this relation that if we introduce the set $\Omega_0(t)=\{p_0>0\}$, then we obtain a free boundary problem of Hele-Shaw type: On $\Omega_0(t)$, we have $n_0=1$ and $-\Delta p_0 = G(p_0)$, whereas $p_0=0$ on $\RR^d\setminus \Omega_0(t)$.
Thus although the pressure law is different, we expect to recover the same free boundary Hele-Shaw model as in \cite{PQV}.

The incompressible limit of the above cell mechanical model for tumor growth with a pressure law given by $\Pi(n)=\frac{\gamma}{\gamma-1} n^{\gamma-1}$ has been investigated in \cite{PQV} and in \cite{PQTV} when taking into account active motion of cells. In \cite{PV}, the case with viscosity, where the Darcy's law is replaced by the Brinkman's law, is studied.
We mention also the recent works \cite{KP,MPQ} where the incompressible limit with more general assumptions on the initial data has been investigated.
However, in all these mentionned works the pressure law do not prevent the non-overlapping of cells. Up to our knowledge, this work is the first attempt to extend the previous result with this constraint, i.e. with a singular pressure law as given by \eqref{eq:p}.

The outline of the paper is the following. In the next section we give the statement of the main result in Theorem \ref{TH1}, which is the convergence when $\epsilon$ goes to $0$ of the mechanical model \eqref{eq:n}--\eqref{eq:p} towards the Hele-Shaw free boundary system.
The rest of the paper is devoted to the proof of this result.
First, in section \ref{sec:estim} we establish some a priori estimate allowing to obtain space compactness.
Then, section \ref{sec:tcompact} is devoted to the study of the time compactness. 
Thanks to compactness results, we can pass to the limit $\epsilon\to 0$ in system \eqref{eq:n}--\eqref{eq:p} in section \ref{sec:conv}, up to the extraction of a subsequence.
Finally the proof of the complementary relation \eqref{eq:p0} is performed in section \ref{sec:compl}.

\section{Main result}

The aim of this paper is to establish the incompressible limit $\epsilon \to 0$
of the cell mechanical model with non-overlapping constraint \eqref{eq:n}--\eqref{eq:p}.
Before stating our main result, we list the set of assumptions that we use 
on the growth fonction and on the initial data.
For the growth function, we assume
\begin{equation}\label{hypG}
  \left\{
      \begin{aligned}
         &\exists\, G_m>0, \quad \| G \|_{\infty} \leq G_m,\\
        & G' <0, \quad \mbox{ and }\  \exists \gamma>0, \quad \min_{[0,P_M]} |G'| = \gamma,\\
         &\exists\, P_M>0, \quad G(P_M)=0.
      \end{aligned}
    \right.
\end{equation}
The quantity $P_M$, for which the growth stops, is commonly called the homeostatic pressure \cite{Prost}. 
This set of assumptions on the growth function is quite similar to the one in \cite{PQV}, except for the bound on the growth term which is needed here due to the singularity in the pressure law.

For the initial data, we assume that there exists $\epsilon_0>0$ such that
for all $\epsilon\in (0,\epsilon_0)$,
\begin{equation}\label{hypini}
  \left\{
      \begin{aligned}
       & 0 \leq n^{ini}_{\epsilon}, \qquad  \ p^{ini}_{\epsilon}:= \epsilon \frac{n_\epsilon^{ini}}{\epsilon+n_\epsilon^{ini}} \leq P_M, \\
       &  \| \partial_{x_i} n^{ini}_{\epsilon} \|_{L^1(\RR^d)} \leq C, \qquad i=1,...,d,\\       & \exists \, n^{ini}_0 \in L^1_+(\RR^d), \quad \|n^{ini}_\epsilon - n^{ini}_0\|_{L^1(\RR^d)} \to 0 \mbox{ as  }\epsilon\to 0,  \\
       & \exists\, K \subset \RR^d,\  K \mbox{ compact}, \quad \forall\, \epsilon\in(0,\epsilon_0), \ \mbox{supp } n_\epsilon^{ini} \subset K.  \\
      \end{aligned}
    \right.
\end{equation}
Notice that this set of assumptions imply that $n_\epsilon^{ini}$ is uniformly bounded in $W^{1,1}(\RR^d)$.

We are now in position to state our main result.
\begin{thm}\label{TH1}
Let $T>0$, $Q_T=(0,T)\times\RR^d$. 
Let $G$ and $(n^{ini}_{\epsilon})$ satisfy assumptions \eqref{hypG} and \eqref{hypini} respectively. 
After extraction of subsequences, both the density $n_{\epsilon}$ and the pressure $p_{\epsilon}$ converge strongly in $L^1(Q_T)$ as $\epsilon \rightarrow 0$ to the limit $n_0 \in C([0,T];L^1(\RR^d))\cap BV(Q_T)$ and
$p_0 \in BV(Q_T)\cap L^2([0,T];H^1(\RR^d))$, which satisfy
\begin{align}
\label{boundn0p0}
& 0 \leq n_0 \leq 1, \quad 0 \leq p_0 \leq P_M,  \\
\label{eqn0}
& \partial_t n_0 - \Delta p_0 = n_0 G(p_0), \mbox{ in } \mathcal{D}'(Q_T),
\end{align}
and
\begin{align}
\label{eq2n0}
\partial_t n_0 - \nabla \cdot ( n_0 \nabla p_0) = n_0 G(p_0), & \quad \text{ in } \mathcal{D}'(Q_T).
\end{align}
Moreover, we have the relation
\begin{align}\label{n0p0}
(1-n_0)p_0=0,
\end{align}
and the complementary relation
\begin{align}\label{compl}
 p_0^2(\Delta p_0 + G(p_0)) =0, & \quad \text{ in } \mathcal{D}'(Q_T).
\end{align}
\end{thm}
This result extends the one in \cite{PQV} to singular pressure laws with non-overlapping constraint. We notice that we recover the same limit model whose uniqueness has already been stated in \cite[Theorem 2.4]{PQV}.

Although our proof follows the idea in \cite{PQV}, several technical difficulties must be overcome due to the singularity of the pressure law. Indeed, we first recall that with the choice $\Pi(n)=\frac{\gamma}{\gamma-1} n^{\gamma-1}$, equation \eqref{eq:n} may be rewritten as the porous medium equation $\pa_t n + \Delta n^\gamma = n G(\Pi(n))$. A lot of estimates are known and well established for this equation (see \cite{Vazquez}), in particular a semiconvexity estimate is used in \cite{PQV} which allows to obtain estimate on the time derivative and thus compactness. 
With our choice of pressure law, \eqref{eq:n} should be consider as a fast diffusion equation. Thus we have first to state a comparison principle to obtain a priori estimates (see Lemma \ref{lem:estim}). Unlike in \cite{PQV}, we may not use a semiconvexity estimate to obtain estimate on the time derivative. To do so, we use regularizing effects (see section \ref{sec:tcompact}). Then the convergence proof has to be adapted for these new estimates.

\begin{figure}
\includegraphics[width=0.49\linewidth]{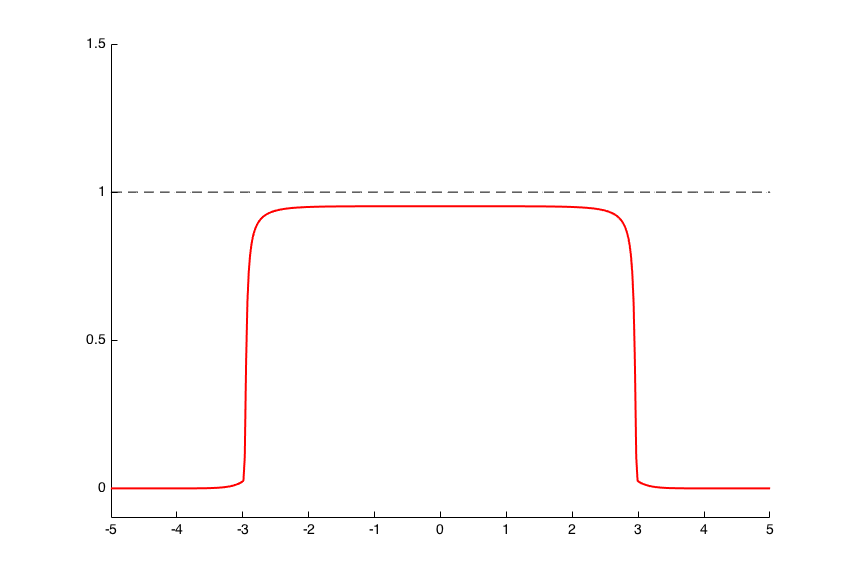}
\includegraphics[width=0.49\linewidth]{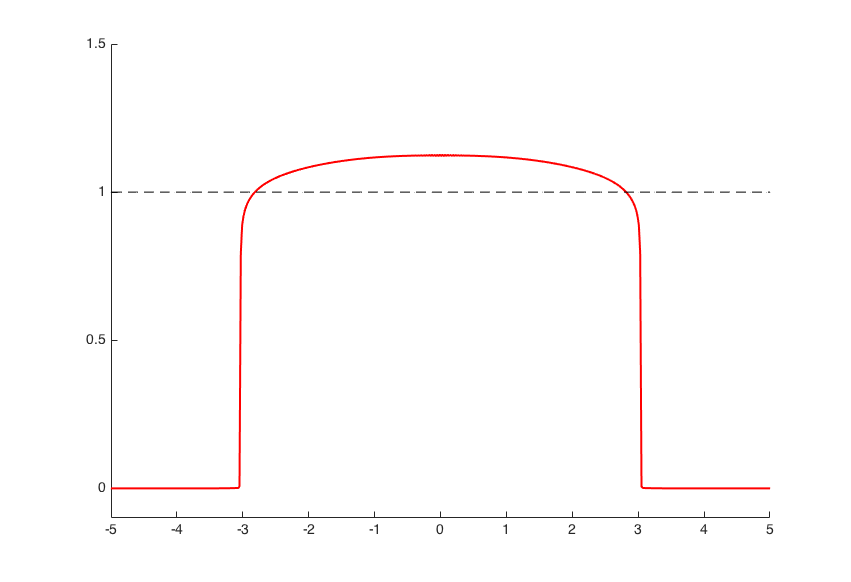}
\caption{Comparison between numerical solutions computed with two different pressure laws. The red line correspond to the cell density $n$ solving \eqref{eq:n}, the dashed line correspond to the constant value $1$. On the left, the pressure law is $p=P(n) = 0.5 \frac{n}{1-n}$. On the right, the pressure law is $p=\Pi(n)=\frac{\gamma}{\gamma-1} n^\gamma$ with $\gamma=20$.}
\label{fig1}
\end{figure}

Finally, we illustrate the comparison between the two pressure laws $P$ and $\Pi$ by a numerical simulation. We display in Figure \ref{fig1} the density computed thanks to a discretization with an upwind scheme of \eqref{eq:n}. In Figure \ref{fig1}-left, the pressure law is $p=P(n)=\epsilon\frac{n}{1-n}$ as in \eqref{eq:p} with $\epsilon=0.5$. In Figure \ref{fig1}-right, the pressure law is $p=\Pi(n)=\frac{\gamma}{\gamma-1} n^\gamma$ with $\gamma=20$.
We take $G(p)= 10(10-p)_+$ as growth function (which satisfies obviously assumption \eqref{hypG} with $P_M=10$).
The dashed lines in these plots correspond to the constant value $1$.
As expected, we observe that the density $n$ is bounded by $1$ in the case of the pressure law $P$ whereas it takes values greater than $1$ for the pressure law $\Pi$. This observation illustrates the fact that the choice of the pressure law $\Pi$ does not prevent from overlapping.

\section{A priori estimates}\label{sec:estim}
\subsection{Nonnegativity principle}

The following Lemma establishes the nonnegativity of the density.
\begin{lem}\label{lem:nonneg}
Let $(n_\epsilon,p_\epsilon)$ be a solution to \eqref{eq:n} such that $n_\epsilon^{ini}\geq 0$ 
and $\|G\|_\infty \leq G_m <\infty$.
Then, for all $t\geq 0$, $n_\epsilon(t)\geq 0$.
\end{lem}
\begin{proof}
We have the equation
$$\partial_t n_{\epsilon} - \nabla \cdot (n_{\epsilon} \nabla p_{\epsilon}) = n_{\epsilon} G(p_{\epsilon}). $$
We use the Stampaccchia method. We multiply by $\mathbf{1}_{n_{\epsilon}<0}$, then
using the notation $|n|_- = \max(0,-n)$ for the negative part, we get
$$\frac{d}{dt} |n_{\epsilon}|_{-} -\nabla \cdot (|n_{\epsilon}|_{-} \nabla p_{\epsilon}) = |n_{\epsilon}|_{-} G(p_{\epsilon}). $$
We integrate in space, using assumption \eqref{hypG}, we deduce
$$\frac{d}{dt} \int_{\RR^d} |n_{\epsilon}|_{-}dx \leq \int_{\RR^d} |n_{\epsilon}|_{-} G(p_{\epsilon})dx \leq G_m\int_{\RR^d} |n_{\epsilon}|_{-}dx. $$
So, after a time integration
$$  \int_{\RR^d} |n_{\epsilon}|_{-}\,dx \leq e^{G_mt}  \int_{\RR^d} |n_{\epsilon}^{ini}|_{-}\,dx. $$
With the initial condition $n_{\epsilon}^{ini} \geq 0$, we deduce $n_{\epsilon} \geq 0$.
\end{proof}

\subsection{A priori estimates}

In order to use compactness results, we need first to find a priori estimates on the pressure and the density.
We first observe that we may rewrite system \eqref{eq:n} as, by using \eqref{eq:p},
\begin{equation}\label{eq:nH}
\partial_t n_{\epsilon} -\Delta H(n_{\epsilon})= n_{\epsilon} G(P(n_{\epsilon})),
\end{equation}
with  $H(n)=\int_0^{n} u P'(u)du = P(n)-\epsilon \ln (P(n)+\epsilon) + \epsilon \ln \epsilon$.

\begin{lem}\label{lem:estim}
Let us assume that \eqref{hypG} and \eqref{hypini} hold. 
Let $(n_\epsilon,p_\epsilon)$ be a solution to \eqref{eq:nH}--\eqref{eq:p}.
Then, for all $T>0$, 
we have the uniform bounds in $\epsilon\in(0,\epsilon_0)$,
\begin{align*}
& 0 \leq n_{\epsilon} \in L^\infty([0,T];L^1\cap L^\infty(\RR^d)); \\
& 0\leq p_\epsilon \leq P_M, \qquad 0\leq n_\epsilon \leq \frac{P_M}{P_M+\epsilon} \leq 1.
\end{align*}

More generally, we have the {\bf comparison principle}:  
If $n_\epsilon$, $m_\epsilon$  are respectively subsolution and supersolution to \eqref{eq:nH},
with initial data $n_\epsilon^{ini}$, $m_\epsilon^{ini}$ as in \eqref{hypini} and satisfying $n_\epsilon^{ini}\leq m_\epsilon^{ini}$.
Then for all $t>0$,  $n_\epsilon(t) \leq m_\epsilon(t)$.

Finally, we have that $(n_\epsilon)_\epsilon$ is uniformly bounded in 
$L^\infty([0,T],W^{1,1}(\RR^d))$ and $(p_\epsilon)_\epsilon$ is uniformly bounded in
$L^1([0,T],W^{1,1}(\RR^d))$.

\end{lem}

\begin{proof}
{\bf Comparison principle.}

Let $n_{\epsilon}$ be a subsolution and $m_{\epsilon}$ a supersolution of \eqref{eq:nH}, we have 
$$
\partial_t (n_{\epsilon} -m_{\epsilon})-\Delta (H(n_{\epsilon})-H(m_{\epsilon})) \leq n_{\epsilon} G(P(n_{\epsilon})) -m_{\epsilon}G(P(m_{\epsilon})).
$$
Notice that, since the function $H$ is nondecreasing, the sign of $n_{\epsilon}-m_{\epsilon}$ is the same as the sign of $H(n_{\epsilon})-H(m_{\epsilon})$. Moreover,
$$
\Delta f(y) = f''(y) |\nabla y|^2 +f'(y) \Delta y, 
$$
so for $y=H(n_{\epsilon})-H(m_{\epsilon})$ and $f(y)=y_+$ is the positive part, the so-called Kato inequality reads
$\Delta f(y) \geq f'(y) \Delta y$.
Thus multiplying the latter equation by $\mathbf{1}_{n_\epsilon-m_\epsilon>0}$, we obtain
\begin{align*}
\partial_t |n_{\epsilon} -m_{\epsilon}|_+-\Delta |H(n_{\epsilon})-H(m_{\epsilon})|_+ \leq & 
|n_{\epsilon}-m_\epsilon|_+ G(P(n_{\epsilon}))   \\
& + m_{\epsilon}(G(P(n_{\epsilon}))-G(P(m_{\epsilon}))) \mathbf{1}_{n_\epsilon-m_\epsilon>0}.
\end{align*}
From assumption \eqref{hypG}, we have that $G$ is nonincreasing. Thus, since $n\mapsto P(n)$ is increasing, 
we deduce that the last term of the right hand side is nonpositive.
Since $G$ is uniformly bounded we obtain
$$ 
\partial_t |n_{\epsilon} -m_{\epsilon}|_{+} - \Delta |H(n_{\epsilon})-H(m_{\epsilon})|_{+} \leq G_m|n_{\epsilon}-m_{\epsilon}|_{+}.
$$
After an integration over $\RR^d$,
$$
\frac{d}{dt} \int_{\RR^d} |n_{\epsilon} -m_{\epsilon}|_{+}\,dx \leq G_m \int_{\RR^d} |n_{\epsilon}-m_{\epsilon}|_{+}\,dx.
$$
Then, integrating in time, we deduce
$$
\int_{\RR^d} |n_{\epsilon} -m_{\epsilon}|_{+}\,dx \leq e^{G_m t} \int_{\RR^d} |n^{ini}_{\epsilon} -m^{ini}_{\epsilon}|_{+}\,dx.
$$
Since we have $n^{ini}_{\epsilon}\leq m^{ini}_\epsilon$, we deduce that
for all $t>0$, $|n_\epsilon-m_\epsilon|_+(t) =0$.

{\bf $L^{\infty}$ bounds.}

We define $n_M= \frac{P_M}{\epsilon+P_M}$, such that $p_M=P(n_M)$, then applying the comparison principle with
$m_\epsilon = n_M$, we deduce, using also the assumption
on the initial data \eqref{hypini} that for all $0<\epsilon\leq \epsilon_0$,
$n_{\epsilon}\leq n_M.$ Moreover, since $0$ is clearly a subsolution to \eqref{eq:nH}, we also have 
by the comparison priniciple $n_\epsilon\geq 0$.
Since $n_M\leq 1$, we have $0\leq n_{\epsilon}\leq n_M \leq 1$ which implies
$$ 0 \leq p_{\epsilon} \leq P_M.$$

{\bf $L^1$ bound of $n, p$.}

By nonnegativity, after a simple integration in space of equation \eqref{eq:n}, we deduce
\begin{equation}\label{eqnL1}
\frac{d}{dt} \| n_\epsilon\|_{L^1(\RR^d)} \leq G_m \| n_\epsilon\|_{L^1(\RR^d)},
\end{equation}
where we use \eqref{hypG}. Integrating in time give the $L^1$ bound,
$$
\|n_{\epsilon}\|_{L^1(\RR^d)} \leq e^{G_mt} \|n^{ini}_{\epsilon}\|_{L^1(\RR^d)}. 
$$
Then, using $p_\epsilon= n_\epsilon (\epsilon+p_\epsilon)$ by \eqref{eq:p}, we get from the bound $p_\epsilon\leq P_M$, which has been proved above,
$$
\|p_{\epsilon}\|_{L^1(R^d)} \leq (\epsilon+P_M)\int_{\RR^d} |n_{\epsilon}|dx \leq Ce^{G_m t} \|n^{ini}_{\epsilon}\|_{L^1(\RR^d)}.
$$

{\bf Estimates on the $x$ derivative.}

We derive equation \eqref{eq:nH} with respect to $x_i$ for $i=1,\ldots,d$,
$$ \partial_t \partial_{x_i} n_{\epsilon} -\Delta (H'(n_{\epsilon})\partial_{x_i} n_{\epsilon})= \partial_{x_i} n_{\epsilon} G(p_{\epsilon}) +n_{\epsilon} G'(p_{\epsilon}) \partial_{x_i} p_{\epsilon}.$$
Multiplying by sign$(\partial_{x_i} n_\epsilon)$, we get
$$ \partial_t |\partial_{x_i} n_{\epsilon}| -\Delta (\partial_{x_i}H(n_{\epsilon}))\mbox{sign}(\partial_{x_i} n_{\epsilon}) = |\partial_{x_i} n_{\epsilon}| G(p_{\epsilon}) +n_{\epsilon} G'(p_{\epsilon}) \partial_{x_i} p_{\epsilon} \mbox{sign}(\partial_{x_i} n_{\epsilon}).$$
We can remark that $\mbox{sign}(\partial_{x_i} n_{\epsilon})= \mbox{sign}(\partial_{x_i} H(n_{\epsilon}))$, so, by the same token as above, we have
$$\Delta(\partial_{x_i}H(n_{\epsilon}))\mbox{sign}(\partial_{x_i} n_{\epsilon}) \geq \Delta (|\partial_{x_i}H(n_{\epsilon})|).$$
Moreover, $\mbox{sign}(\partial_{x_i} n_{\epsilon})=\mbox{sign}(\partial_{x_i} p_{\epsilon})$, thus 
$\partial_{x_i} p_{\epsilon} \mbox{sign}(\partial_{x_i} n_{\epsilon}) = |\partial_{x_i} p_\epsilon|$. 
By assumption \eqref{hypG}, we know that
$$ G'(p_{\epsilon}) \leq - \gamma<0, $$
we deduce
$$ \partial_t |\partial_{x_i} n_{\epsilon}| -\Delta (|\partial_{x_i}H(n_{\epsilon})|) \leq  |\partial_{x_i} n_{\epsilon}| G_{m} - \gamma n_{\epsilon} |\partial_{x_i} p_{\epsilon}|.$$
After an integration in time and space,
\begin{equation}\label{estimdxn}
\|\partial_{x_i} n_{\epsilon}\|_{L^1(\RR^d)} + \gamma \int_0^t \int_{\RR^d} n_{\epsilon} |\partial_{x_i} p_{\epsilon}|\,dxds \leq e^{G_m t} \| \partial_{x_i} n_{\epsilon}^{ini}\|_{L^1(\RR^d)}.
\end{equation}

This latter inequality provides us with a uniform bound on the space derivative of $n_\epsilon$ in $L^1$. Then
$$ 
\| \partial_{x_i} p_{\epsilon}\|_{L^1(\RR^d)} 
= \int_{\RR^d} |\partial_{x_i} p_{\epsilon}|dx = \int_{\RR^d} \frac{\epsilon}{(1-n_{\epsilon})^2}|\partial_{x_i} n_{\epsilon}|dx.
$$
We split the integral in two: Either $n_{\epsilon}\leq 1/2$ and then $\frac{\epsilon}{(1-n_{\epsilon})^2} \leq C$; or $n_{\epsilon}> 1/2$.
\begin{align*} 
\| \partial_{x_i} p_{\epsilon}\|_{L^1(\RR^d)}  &\leq  C \int_{n_{\epsilon}\leq 1/2} |\partial_{x_i} n_{\epsilon}|dx + \int_{n_{\epsilon}> 1/2} |\partial_{x_i} p_{\epsilon}|dx \\
  &\leq  C \int_{n_{\epsilon}\leq 1/2} |\partial_{x_i} n_{\epsilon}|dx + 2 \int_{n_{\epsilon}> 1/2} \frac 12|\partial_{x_i} p_{\epsilon}|dx \\
   &\leq  C e^{G_m t} \int_{n_{\epsilon}\leq 1/2} |\partial_{x_i} n_{\epsilon}^{ini}|dx + 2 \int_{n_{\epsilon}>1/2} n_\epsilon |\partial_{x_i} p_{\epsilon}|dx,
\end{align*}
where we have used the estimate \eqref{estimdxn} for the last inequality.
Then, integrating in time, we deduce, using again the estimate \eqref{estimdxn}
$$
\| \partial_{x_i} p_{\epsilon}\|_{L^1(Q_T)} \leq C' e^{G_m t} \| \partial_{x_i} n_{\epsilon}^{ini}\|_{L^1(\RR^d)}.
$$
It concludes the proof.
\end{proof}

\subsection{Compact support}

The following Lemma proves that assuming that the initial data is compactly supported, then the pressure is compactly supported for any time with a control of the growth of the support.
\begin{lem}[Finite speed of propagation]\label{lem:supp}
Under the same assumptions as in Theorem \ref{TH1}, we have that $\mbox{supp }p_\epsilon \subset B(0,R(t))$ with 
$R(t)\leq 2 \sqrt{C(T+t)}$, where $B(0,R(t))$ is the ball of center $0$ and radius $R(t)$.
\end{lem}
\begin{proof} 
Using the equation on $p_{\epsilon}$ \eqref{eq:p1},
$$
\partial_t p_{\epsilon} - (\frac{p_{\epsilon}^2}{\epsilon}+p_{\epsilon})\Delta p_{\epsilon} -  |\nabla p_{\epsilon}|^2 =(\frac{p_{\epsilon}^2}{\epsilon}+p_{\epsilon}) G(p_{\epsilon}) \leq G_{m}(\frac{p_{\epsilon}^2}{\epsilon}+p_{\epsilon}).
$$
Let us introduce for $C>0$,
$$ \tilde{p}(t,x) = \left(C+\frac{|x|^2}{4(\theta+t)}\right)_+ ,$$ with
$\theta = \frac{d}{4 G_{m}}$. Then $\tilde{p}$ is compactly supported in $B(0,R_\theta(t))$ with 
$R_\theta(t)= 2 \sqrt {C(\theta+t)}.$
We have
$$\partial_t \tilde{p} = \frac{|x|^2}{4(\theta+t)^2} 1_{|x|\leq R_\theta(t)}, \qquad
|\nabla \tilde{p}|^2= \frac{|x|^2}{4(\theta+t)^2} 1_{|x|\leq R_\theta(t)}, $$
and
$$ \Delta \tilde{p} = -\frac{d}{(\theta+t)}, \mbox{ for } |x| < R_\theta(t).$$
Then, for all $t\in [0,\theta]$,
\begin{equation}\label{ineqptilde}
\partial_t \tilde{p}- (\frac{\tilde{p}^2}{\epsilon}+\tilde{p})\Delta \tilde{p} 
-|\nabla \tilde{p}|^2 -G_{m}(\frac{\tilde{p}^2}{\epsilon}+\tilde{p}) 
= (\frac{\tilde{p}^2}{\epsilon}+\tilde{p}) (\frac{d}{(\theta+t)}-G_{m}) \geq 0.
\end{equation}
In other words, $\tilde{p}$ is a supersolution for the equation for the pressure. 
Let us show that it implies that $p\leq \tilde{p}$. We define 
$\tilde{n} = \frac{\tilde{p}}{\epsilon+\tilde{p}}= N(\tilde{p})$. We know that
$$ N'(\tilde{p})= \frac{\epsilon}{(\epsilon+\tilde{p})^2} >0. $$
Then, on the one hand, multiplying \eqref{ineqptilde} with by $N'(\tilde{p})$ we get
$$ \partial_t \tilde{n}-\nabla.(\tilde{n} \nabla \tilde{p}) -G_{m}\tilde{n} \geq 0.$$
On the other hand, from \eqref{eq:n},
$$  \partial_t n_{\epsilon}-\nabla.(n_{\epsilon} \nabla p_{\epsilon}) \leq G_{m}n_{\epsilon}.$$
By the comparison principle (see Lemma \ref{lem:estim}), we have
$$ n^{ini}_{\epsilon} \leq \tilde{n}^{ini} \Rightarrow n_{\epsilon} \leq \tilde{n}. $$
Thus, for all $t\in [0,\theta]$,
$$ p^{ini}_{\epsilon} \leq \tilde{p}(t=0) \Rightarrow p_{\epsilon} \leq \tilde{p}. $$
and $p_\epsilon(t)$ is compactly supported in $B(0,R_\theta(t))$ provided we choose $C$ large enough such
that $p_\epsilon^{ini}(x)\leq \tilde{p}(t=0,x)$, which can be done thanks to our assumption on the initial data \eqref{hypini}.

Since $p_\epsilon$ is uniformly bounded in $L^\infty$, we may iterate the process on $[\theta,2\theta]$.
After several iterations, we reach the time $T$ and prove the result on $[0,T]$.
\end{proof}

 \subsection{$L^2$ estimate for $\nabla p$}

In the following Lemma, we state a uniform $L^2$ estimate on the gradient of the pressure.
\begin{lem}[$L^2$ estimate for $\nabla p$]\label{lem:L2dp}
Under the same assumptions as in Theorem \ref{TH1}, we have a uniform bound on $\nabla p_\epsilon$ in $L^2(Q_T)$.
\end{lem}
\begin{proof}
For a given function $\psi$ we have, multiplying \eqref{eq:n} by $\psi(n_\epsilon)$,
$$
\partial_t n_{\epsilon} \psi(n_\epsilon) -\nabla(n_{\epsilon} \nabla p_{\epsilon})\psi(n_{\epsilon})= n_{\epsilon} G(p_{\epsilon})\psi(n_\epsilon).
$$
Let $\Psi$ be an antiderivative of $\psi$, we have thanks to an integration by parts
$$ 
\frac{d}{dt} \int_{\RR^d} \Psi(n_{\epsilon})\,dx + \int_{\RR^d} n_{\epsilon} \nabla n_{\epsilon}\cdot \nabla p_{\epsilon}\psi'(n_{\epsilon})\,dx = \int_{\RR^d} n_{\epsilon} G(p_{\epsilon})\psi(n_{\epsilon})\,dx.
$$
We choose $\psi$ such as $n_{\epsilon} \nabla n_{\epsilon}\cdot \nabla p_{\epsilon}\psi'(n_{\epsilon})= |\nabla p_{\epsilon}|^2$,
i.e. $n_{\epsilon}\psi'(n_{\epsilon})=p'(n_{\epsilon})$.
After straightforward computations, we find
$\psi(n)= \epsilon (\ln(n)-\ln(1-n)+\frac{1}{1-n})$ and 
$\Psi(n)= \epsilon n(\ln(n)-\ln(1-n))$.
It gives 
\begin{align*}
&\frac{d}{dt} \int_{\RR^d} \epsilon n_{\epsilon} \ln\Big(\frac{n_{\epsilon}}{1-n_{\epsilon}}\Big)\,dx 
+ \int_{\RR^d} |\nabla p_{\epsilon}|^2 dx  
\leq G_{m} \int_{\RR^d} \epsilon n_{\epsilon}\left|\ln(n_{\epsilon})-\ln(1-n_{\epsilon})+\frac{1}{1-n_{\epsilon}}\right|\,dx.
\end{align*}
We integrate in time, using also the expression of $p_\epsilon$ in \eqref{eq:p},
\begin{align*}
&\int_{\RR^d} \epsilon n_{\epsilon} \ln\Big(\frac{p_{\epsilon}}{\epsilon}\Big)\,dx 
-\int_{\RR^d} \epsilon n_{\epsilon}^{ini} \ln\left(\frac{n_{\epsilon}^{ini}}{1-n_{\epsilon}^{ini}}\right)\,dx + \int_0^T \int_{\RR^d} |\nabla p_{\epsilon}|^2\,dxdt   \\
&\leq G_{m} \int_0^T \int_{\RR^d} \left(\epsilon n_{\epsilon} \Big|\ln\Big(\frac{p_{\epsilon}}{\epsilon}\Big)\Big|+p_{\epsilon}\right) \,dx.
\end{align*}
Then, to have a bound on the $L^2$-norm of $\nabla p_\epsilon$, it suffices to prove a uniform control on 
$ \int_{\RR^d} \epsilon n_{\epsilon} |\ln(\frac{p_{\epsilon}}{\epsilon})|dx$. We have
$$
\int_{\RR^d} \epsilon n_{\epsilon} |\ln\big(\frac{p_{\epsilon}}{\epsilon}\big)|\,dx \leq \int_{\RR^d} \epsilon n_{\epsilon} |\ln p_{\epsilon}| \,dx  + \epsilon \ln(\epsilon) \int_{\RR^d} n_{\epsilon}\,dx.
$$
The second term of the right hand side is small when $\epsilon$ is small thanks to the $L^1$ bound on $n_\epsilon$, 
thus it is uniformly bounded.
Using the expression of $p_\epsilon$ in \eqref{eq:p}, we get
$$
\int_{\RR^d} \epsilon n_\epsilon |\ln(\frac{p_{\epsilon}}{\epsilon})|\,dx \leq  \int_{\RR^d} (1-n_{\epsilon})p_{\epsilon} |\ln p_{\epsilon}|\,dx + C.
$$
Then, since $0\leq p_\epsilon \leq P_M$ and since $x\mapsto x|\ln x|$ is uniformly bounded on $[0,P_M]$, we get
$$  \int_{\RR^d} (1-n_{\epsilon})p_{\epsilon} |\ln(p_{\epsilon})|\,dx \leq C \int_{\RR^d} \mathbf{1}_{p_{\epsilon}>0} \,dx.$$
We conclude thanks to Lemma \ref{lem:supp}, which provides a uniform control on the support of $p_\epsilon$.
\end{proof}

\section{Regularizing effect and time compactness}\label{sec:tcompact}

As already noticed in \cite{PQTV}, regularizing effects, similar to the ones observed for the heat equation \cite{AB,CP},
allow to deduce estimates on the time derivatives.

\begin{lem}\label{lem:regul}
Under the assumptions \eqref{hypG} and \eqref{hypini}, the weak solution $(\rho_k,p_k)$ satisfies the
estimates
$$
\partial_t p_\epsilon \geq - \frac{\kappa p_\epsilon}{t}, \qquad \partial_t n_\epsilon \geq -\frac{\kappa n_\epsilon}{t},
$$
for a large enough (independent of $\epsilon$) constant $\kappa$.
\end{lem}

\begin{proof}
Let us denote $w_\epsilon=\Delta p_\epsilon + G(p_\epsilon)$, the equation on the pressure \eqref{eq:p1} reads
\begin{equation}\label{eq:preg}
\partial_t p_\epsilon = \left(\frac{p_\epsilon^2}{\epsilon}+p_\epsilon\right) w_\epsilon + |\nabla p_\epsilon|^2.
\end{equation}
The proof is divided into several steps. We first find a lower bound for $w_\epsilon$ by using the comparison principle.
Then we deduce estimates on the density and on the pressure. \\

{\it 1st step.} Thanks to \eqref{eq:preg}, we deduce an equation satisfied by $w_\epsilon$.
On the one hand, by multiplying \eqref{eq:preg} by $G'(p_\epsilon)$, we deduce, since $G$ is decreasing from \eqref{hypG}
\begin{equation}\label{ineq:Gp}
\partial_t G(p_\epsilon) \geq G'(p_\epsilon) \big(\frac{p_\epsilon^2}{\epsilon}+p_\epsilon\big) w_\epsilon + 2 \nabla G(p_\epsilon)\cdot \nabla p_\epsilon.
\end{equation}
On the other hand, we have
\begin{align*}
\partial_t \Delta p_\epsilon = & \Delta w_\epsilon \big(\frac{p_\epsilon^2}{\epsilon}+p_\epsilon\big) + 2 \nabla\big(\frac{p_\epsilon^2}{\epsilon}+p_\epsilon\big)
\cdot \nabla w_\epsilon + w_\epsilon \Delta (\frac{p_\epsilon^2}{\epsilon}+p_\epsilon)  \\
& + 2 \nabla p_\epsilon\cdot \nabla(\Delta p_\epsilon) + 2 \sum_{i,j=1}^d (\partial_{x_ix_j} p_\epsilon)^2   \\
\geq &\Delta w_\epsilon (\frac{p_\epsilon^2}{\epsilon}+p_\epsilon) + 2 \nabla(\frac{p_\epsilon^2}{\epsilon}+p_\epsilon)
\cdot \nabla w_\epsilon + w_\epsilon \Delta (\frac{p_\epsilon^2}{\epsilon}+p_\epsilon)  \\
& + 2 \nabla p_\epsilon\cdot \nabla(\Delta p_\epsilon) + \frac{2}{d} (\Delta p_\epsilon)^2.
\end{align*}
Thus, with \eqref{ineq:Gp}, we deduce that $w_\epsilon=\Delta p_\epsilon+G(p_\epsilon)$ satisfies
\begin{align*}
\partial_t w_\epsilon \geq &\Delta w_\epsilon (\frac{p_\epsilon^2}{\epsilon}+p_\epsilon) + 2 \nabla(\frac{p_\epsilon^2}{\epsilon}+p_\epsilon)
\cdot \nabla w_\epsilon + w_\epsilon \Big(\Delta p_\epsilon (\frac{2p_\epsilon}{\epsilon}+1) + \frac{2}{\epsilon} |\nabla p_\epsilon|^2  \\
& + (\frac{p_\epsilon^2}{\epsilon}+p_\epsilon)G'(p_\epsilon)\Big)
 + 2 \nabla p_\epsilon\cdot \nabla w_\epsilon + \frac{2}{d} (\Delta p_\epsilon)^2.
\end{align*}
By definition of $w_\epsilon$, we have $(\Delta p_\epsilon)^2 \geq w_\epsilon^2 - 2 G(p_\epsilon) w_\epsilon$. Thus we deduce that 
\begin{equation}\label{eq:w}
\partial_t w_\epsilon \geq \mathcal{F}(w_\epsilon),
\end{equation}
where we have used the notation
\begin{align}
\mathcal{F}(w) := & \Delta w (\frac{p_\epsilon^2}{\epsilon}+p_\epsilon) + 2 \nabla(\frac{p_\epsilon^2}{\epsilon}+2p_\epsilon)
\cdot \nabla w + \frac{2}{\epsilon} |\nabla p_\epsilon|^2 w + w^2 (\frac{2p_\epsilon}{\epsilon}+1+ \frac{2}{d})  \nonumber \\
& - w\Big(G(p_\epsilon)(\frac{2 p_\epsilon}{\epsilon} +1 + \frac{4}{d}) - (\frac{p_\epsilon^2}{\epsilon} +p_\epsilon) G'(p_\epsilon)\Big).
\label{eq:F}
\end{align}

Following an idea of \cite{CP} which has been generalized in \cite{PQTV}, we introduce the function 
\begin{equation}\label{def:W}
W(t,x) = -\frac{h(p_\epsilon(t,x))}{t},
\end{equation}
where the function $h$ will be defined later such that $W$ is a subsolution for \eqref{eq:w}.
We compute
\begin{align*}
&\partial_t W = \frac{W^2}{h(p_\epsilon)} - \frac{h'(p_\epsilon)}{t} \partial_t p_\epsilon,  \\
&\nabla W = -\frac{h'(p_\epsilon)}{t} \nabla p_\epsilon, \qquad 
\Delta W = -\frac{h'(p_\epsilon)}{t} \Delta p_\epsilon - \frac{h''(p_\epsilon)}{t} |\nabla p_\epsilon|^2.
\end{align*}
Using again equation \eqref{eq:preg}, we have
\begin{align}
\partial_t W & = \frac{W^2}{h(p_\epsilon)} - \frac{h'(p_\epsilon)}{t}(\frac{p_\epsilon^2}{\epsilon}+p_\epsilon) \Delta p_\epsilon 
- \frac{h'(p_\epsilon)}{t}(\frac{p_\epsilon^2}{\epsilon}+p_\epsilon) G(p_\epsilon) - \frac{h'(p_\epsilon)}{t} |\nabla p_\epsilon|^2  \nonumber \\
& = \frac{W^2}{h(p_\epsilon)} + (\frac{p_\epsilon^2}{\epsilon}+p_\epsilon) \Delta W + \frac{h''(p_\epsilon)}{t} |\nabla p_\epsilon|^2 (\frac{p_\epsilon^2}{\epsilon}+p_\epsilon)
- \frac{h'(p_\epsilon)}{t} |\nabla p_\epsilon|^2   \nonumber \\
& \quad  - \frac{h'(p_\epsilon)}{t}  (\frac{p_\epsilon^2}{\epsilon}+p_\epsilon) G(p_\epsilon).
\label{eqW1}
\end{align}
By definition of $\mathcal{F}(W)$ in \eqref{eq:F}, we deduce with \eqref{eqW1},
\begin{align*}
\partial_t W = & \mathcal{F}(W) + 4(\frac{p_\epsilon}{\epsilon}+1) |\nabla p_\epsilon|^2 \frac{h'(p_\epsilon)}{t}
+ \frac{2}{\epsilon} \frac{h(p_\epsilon)}{t} |\nabla p_\epsilon|^2   \\
& + W^2\Big(\frac{1}{h(p_\epsilon)} -\frac{2p_\epsilon}{\epsilon} - 1 -\frac{2}{d}\Big)
+ \frac{h''(p_\epsilon)}{t} |\nabla p_\epsilon|^2 (\frac{p_\epsilon^2}{\epsilon}+p_\epsilon)
- \frac{h'(p_\epsilon)}{t} |\nabla p_\epsilon|^2   \\
& - \frac{h'(p_\epsilon)}{t}  (\frac{p_\epsilon^2}{\epsilon}+p_\epsilon) G(p_\epsilon) + W \Big(G(p_\epsilon)(\frac{2p_\epsilon}{\epsilon}+1+\frac{4}{d}) - (\frac{p_\epsilon^2}{\epsilon}+ p_\epsilon) G'(p_\epsilon)\Big).
\end{align*}
We may rearrange it into
\begin{align}
\partial_t W = & \mathcal{F}(W) + W^2\Big(\frac{1}{h(p_\epsilon)} -\frac{2p_\epsilon}{\epsilon} - 1 -\frac{2}{d}\Big)
+ \frac{|\nabla p_\epsilon|^2}{t} \Big(\big(h(p_\epsilon)(\frac{p_\epsilon^2}{\epsilon}+p_\epsilon)\big)''+h'(p_\epsilon)\Big)    \nonumber \\
& - \frac{h'(p_\epsilon)}{t}  (\frac{p_\epsilon^2}{\epsilon}+p_\epsilon) G(p_\epsilon) + W \Big(G(p_\epsilon)(\frac{2p_\epsilon}{\epsilon}+1+\frac{4}{d}) - (\frac{p_\epsilon^2}{\epsilon}+ p_\epsilon) G'(p_\epsilon)\Big).
\label{eqWf}
\end{align}

Let us choose
\begin{equation}\label{def:h}
h(p) = \frac{ \kappa \epsilon}{p+\epsilon},
\end{equation}
where $\kappa>0$ is chosen large enough (independent of $\epsilon$) such that
$$
\frac{1}{h(p_\epsilon)} = \frac{p_\epsilon+\epsilon}{ \kappa \epsilon} \leq \frac{2 p_\epsilon}{\epsilon}+1+\frac{2}{d}.
$$
Thanks to this choice, we have
$$
\big(h(p_\epsilon)(\frac{p_\epsilon^2}{\epsilon}+p_\epsilon)\big)''+h'(p_\epsilon) = -\frac{\kappa \epsilon}{(p_\epsilon+\epsilon)^2} \leq 0,
$$
and 
$$
- \frac{h'(p_\epsilon)}{t}  (\frac{p_\epsilon^2}{\epsilon}+p_\epsilon) = W \frac{p_\epsilon}{\epsilon}.
$$
Finally, we obtain from \eqref{eqWf}
$$
\partial_t W \leq \mathcal{F}(W) + W \Big(G(p_\epsilon)(\frac{p_\epsilon}{\epsilon}+1+\frac{4}{d}) - (\frac{p_\epsilon^2}{\epsilon}+ p_\epsilon) G'(p_\epsilon)\Big) \leq \mathcal{F}(W),
$$
where we use the fact that by definition \eqref{def:W} we have $W\leq 0$ (recalling also that $G$ is decreasing by assumption \eqref{hypG}).

Thus, by the sub- and super-solution technique, we deduce, using also \eqref{eq:w} that 
\begin{equation}\label{estimw}
w_\epsilon\geq W = - \frac{\kappa \epsilon}{t(p_\epsilon+\epsilon)}.
\end{equation}
 \\

{\it 2nd step.} 
Using again equation \eqref{eq:preg}, we get from \eqref{estimw}
$$
\partial_t p_\epsilon \geq (\frac{p_\epsilon^2}{\epsilon}+p_\epsilon) W = - \frac{\kappa p_\epsilon}{t},
$$
which is the first inequality of Lemma \ref{lem:regul}.
Finally, by definition \eqref{eq:p}, we have also $n_\epsilon = \frac{p_\epsilon}{p_\epsilon+\epsilon}$.
Thus 
\begin{align*}
\partial_t n_\epsilon & = \frac{\epsilon}{(p_\epsilon+\epsilon)^2} \partial_t p_\epsilon 
\geq -\frac{\kappa \epsilon p_\epsilon}{t(p_\epsilon+\epsilon)^2}
= -\frac{\kappa n_\epsilon(1-n_\epsilon)}{t},
\end{align*}
where we use the definition \eqref{eq:p} for the last identity. We conclude easily the proof.
\end{proof}

Thanks to this latter Lemma, we may deduce uniform estimates on the time derivative of $n_\epsilon$ and $p_\epsilon$.
\begin{lem}\label{estimdtn}
For any $\tau>0$, we have that $\partial_t n_\epsilon$ is uniformly bounded in $L^\infty([\tau,T];L^1(\RR^d))$ and 
$\partial_t p_\epsilon$ is uniformly bounded in $L^1([\tau,T]\times\RR^d)$.
\end{lem}
\begin{proof}
We use the equality $|\partial_t n_\epsilon| = \partial_t n_\epsilon + 2 |\partial_t n_\epsilon|_-$, where we recall that $|\cdot|_-$
denotes the negative part.
Thus 
\begin{align*}
\|\partial_t n_\epsilon\|_{L^1(\RR^d)} & = \frac{d}{dt} \int_{\RR^d} n_\epsilon \,dx + 2 \int_{\RR^d} |\partial_t n_\epsilon|_- \,dx  \\
& \leq \Big(G_m + \frac{2 \kappa}{t}\Big) \|n_\epsilon\|_{L^1(\RR^d)},
\end{align*}
where we have used equation \eqref{eqnL1} to bound the first term and Lemma \ref{lem:regul} for the second term.
By the same token, we have
\begin{align*}
\|\partial_t p_\epsilon\|_{L^1([\tau,T]\times\RR^d)} & = \int_\tau^T \frac{d}{dt} \int_{\RR^d} p_\epsilon \,dx 
+ 2 \int_\tau^T\int_{\RR^d} |\partial_t p_\epsilon|_- \,dx    \\
& \leq \|p_\epsilon(T)\|_{L^1(\RR^d)} + \|p_\epsilon\|_{L^\infty([\tau,T];L^1(\RR^d))}  2 \kappa \ln(T/\tau).
\end{align*}
We conclude the proof thanks to the estimates on $n_\epsilon$ and $p_\epsilon$ in $L^1\cap L^\infty$ obtained in Lemma \ref{lem:estim}.
\end{proof}

\section{Convergence}\label{sec:conv}

This section is devoted to the proof of Theorem \ref{TH1} apart from the complementary relation \eqref{compl} which is postponed to the next section.
 
Since the sequences $(n_{\epsilon})_{\epsilon}$ and $(p_{\epsilon})_{\epsilon}$ are bounded in $W^{1,1}_{loc}(Q_T)$, due to Lemma \ref{lem:estim} and \ref{estimdtn}, we may apply the Helly theorem and recover strong convergence in $L^1_{loc}(Q_T)$, up to an extraction. If we want to extend this local convergence to a global convergence in $L^1(Q_T)$ we need to prove that we can control the mass in an initial strip and in the tail.
Indeed, let $\epsilon,\epsilon' >0$, $R >0$, $\tau >0$
\begin{align*}
\| n_{\epsilon}-n_{\epsilon'} \|_{L^1(Q_T)} = &  \int_0^T \int_{\RR^d} |n_{\epsilon}(t,x)-n_{\epsilon'}(t,x)| dx dt \\
 \leq & \int_{\tau}^T \int_{B(0,R)} |n_{\epsilon}(t,x)-n_{\epsilon'}(t,x)| dx dt \\
 & + \int_{\tau}^T \int_{\RR^d \setminus { B(0,R)}} |n_{\epsilon}(t,x)-n_{\epsilon'}(t,x)| dx dt  \\
 & + \int_0^{\tau} \int_{\RR^d} |n_{\epsilon}(t,x)-n_{\epsilon'}(t,x)| dx dt.
\end{align*}
Since we have strong convergence of $n_{\epsilon}$ in $L^1_{loc}(Q_T)$,
$$ \int_{\tau}^T \int_{B(0,R)} |n_{\epsilon}(t,x)-n_{\epsilon'}(t,x)| \,dxdt \underset{\epsilon\to 0}{\longrightarrow} 0.$$
Then we have to control the two other terms in the right hand side. 

 The control of the initial strip comes from the $L^1$ estimate of $n$,
 $$ \int_0^{\tau} \int_{\RR^d} |n_{\epsilon}(t,x)-n_{\epsilon'}(t,x)| dx dt \leq \int_0^{\tau} \Big(\|n_{\epsilon}(t,x)\|_{L^1(\RR^d)}+\|n_{\epsilon'}(t,x)\|_{L^1(\RR^d)} \Big) dt \underset{\tau \to 0}{\longrightarrow} 0$$
 
 For the control of the tail we consider $\phi \in C^{\infty}(\RR)$ such that $0 \leq \phi \leq 1$, $\phi(x)=0$ for $|x|<R-1$ and $\phi(x)=1$ for $|x|>R$. We define $\phi_R(x)=\phi(x/R)$. Then
\begin{align*}
 \int_{\tau}^T \int_{\RR^d \setminus{ B(0,R)}} |n_{\epsilon}(t,x)-n_{\epsilon'}(t,x)| dx dt \leq &  \int_{\tau}^T \int_{\RR^d \setminus{ B(0,R)}} |n_{\epsilon}(t,x)-n_{\epsilon'}(t,x)| \phi_R dx dt \\
 \leq &   \int_{\tau}^T \int_{\RR^d \setminus{ B(0,R)}} (n_{\epsilon}(t,x)+n_{\epsilon'}(t,x)) \phi_R \,dxdt,\\
 \end{align*}
where the notation $C$ stand for a generic nonnegative constant.
Moreover, using equation \eqref{eq:nH}, we deduce
\begin{align*}
\frac{d}{dt} \int_{\RR^d} n_{\epsilon} \phi_R\,dx = &
\int_{\RR^d} H(n_\epsilon) \Delta \phi_R \,dx + \int_{\RR^d} n_\epsilon G(p_\epsilon) \phi_R\,dx  \\
\leq & C R^{-2} \|\Delta \phi\|_{L^{\infty}} + G_{m} \int_{\RR^d} n_{\epsilon} \phi_R\, dx.
\end{align*}
Then, integrating on $[0,T]$, we get
\begin{align*}
0 \leq \int_{\RR^d} n_{\epsilon} \phi_R dx \leq & e^{G_{m} T} \left(\int_{\RR^d} n^{ini}_{\epsilon} \phi_R+CR^{-2}T\right) \\
\leq & e^{G_{m} T} \left(\| n^{ini}_{\epsilon}-n^{ini}\|_{L^1(\RR^d)} + \int_{\RR^d} n^{ini} \phi_R\,dx +CR^{-2}T\right).
\end{align*}
By assumption \eqref{hypini}, since the initial data is uniformly compactly supported, we deduce that the right hand side tends to $0$ as $R$ goes to $+\infty$ and ${\epsilon}$ goes to $0$. Then $(n_{\epsilon})_\epsilon$ is a Cauchy sequence in $L^1(Q_T)$. It implies its convergence in  $L^1(Q_T)$. The convergence of the pressure follows from the same kind of computation. The only difference is for the control of the tail and which is directly given by the estimate
$$
0\leq \int_{\RR^d} p_{\epsilon} \phi_R\,dx \leq (\epsilon+P_M)\int_{\RR^d} n_{\epsilon} \phi_R \,dx.
$$

Therefore, we can extract subsequences and pass to the limit in the equation
$$ (1-n_{\epsilon}) p_{\epsilon} = \epsilon n_{\epsilon},$$ which implies
$$ 
(1-n_0) p_0 =0.
$$
This is the relation \eqref{n0p0}.
We can also pass to the limit in the uniform estimate of Lemma \ref{lem:estim} which provides \eqref{boundn0p0} and $n_0, p_0 \in BV(Q_T)$.

\paragraph{Limit model.} 
We first recall that from \eqref{eq:nH}, we have
$$
\partial_t n_\epsilon - \Delta (p_\epsilon - \epsilon \ln(p_\epsilon + \epsilon)) = n_\epsilon G(p_\epsilon).
$$
We get,
$$
\epsilon \ln \epsilon \leq \epsilon \ln(p_\epsilon + \epsilon) \leq \epsilon \ln(P_M + \epsilon).
$$
Thus, the term in the Laplacien converges strongly to $p_0$ as $\epsilon$ goes to $0$. 
Then, thanks to the strong convergence of $n_\epsilon$
and $p_\epsilon$, we deduce that in the sense of distribution $(n_0,p_0)$ satisfies \eqref{eqn0}.
Moreover, due to the uniform estimate on $\nabla p$ in $L^2(Q_T)$ of Lemma \ref{lem:L2dp}, 
we can show, by passing into the limit in a product of a weak-strong convergence, that in the sense of distribution $(n_0,p_0)$ satisfies \eqref{eq2n0}.

\paragraph{Time continuity.} 
Let us define $0<t_1< t_2 \leq T$,  $\eta>0$.
For a given $R>0$, we consider a smooth function $\zeta_R$ on $\RR^d$ such that $0 \leq \zeta_R \leq 1$, $\zeta_R(x)=1$ for $|x|<R-1$ and $\zeta_R(x)=0$ for $|x|>R$. 
We have
 \begin{align*}
 \int_{\RR^d}|n_0(t_2)-n_0(t_1)|\,dx =  \int_{\RR^d}|n_0(t_2)-n_0(t_1)|\zeta_R\,dx + \int_{\RR^d}|n_0(t_2)-n_0(t_1)|(1-\zeta_R)\,dx.
  \end{align*}
We have 
\begin{align*}
  \int_{\RR^d}|n_0(t_2)-n_0(t_1)|(1-\zeta_R)dx \leq \int_{\RR^d}n_0(t_2)(1-\zeta_R)dx +\int_{\RR^d}n_0(t_1)(1-\zeta_R)dx 
\end{align*}
with $1-\zeta_R$ a function which is zero on $B(0,R-1)$. Thus, as for the control of the tail, for $R$ large enough, we have, uniformly for $0<t_1< t_2 \leq T$,
\begin{align*}
  \int_{\RR^d}|n_0(t_2)-n_0(t_1)|(1-\zeta_R)dx \leq \eta.
\end{align*}
In addition, we know from Lemma \ref{lem:regul} (and the $L^\infty$ bound on $n_0$) that $\partial_t n_0 \geq -\frac{C}{t}$, so $\partial_t (n_0+ C\ln(t)) \geq 0$. Then, since $t_1 < t_2$,
\begin{align*}
  \int_{\RR^d} |n_0(t_2)-n_0(t_1)| \zeta_R \,dx \leq & \int_{\RR^d} (n_0(t_2)+ C\ln(t_2)-(n_0(t_1)+ C\ln(t_1))) \zeta_R \,dx  \\
  &+ \int_{\RR^d} C(\ln(t_2)-\ln(t_1)) \zeta_R \,dx \\
  \leq & \int_{t_1}^{t_2} \int_{\RR^d}  \partial_t (n_0+ C\ln(t)) \zeta_R \,dxdt + \int_{\RR^d} C(\ln(t_2)-\ln(t_1)) \zeta_R \,dx.
\end{align*}
Then, using equation \eqref{eqn0} and an integration by parts, we obtain
\begin{align*}
  \int_{\RR^d} |n_0(t_2)-n_0(t_1)| \zeta_R \,dx  \leq & 
\int_{t_1}^{t_2} \int_{\RR^d} \Big(p_0 \Delta \zeta_R + n_0 G(p_0) \zeta_R\Big)\,dxdt  \\
& + 2 \int_{\RR^d} C(\ln(t_2)-\ln(t_1)) \zeta_R \,dx  \\
  \leq & C(t_2-t_1)( ||\Delta \zeta_R||_{\infty}+1) +  2C(\ln(t_2)-\ln(t_1)) \int_{\RR^d} \zeta_R \,dx.
\end{align*}
Then we can choose $(t_1,t_2)$ close enough such that
$$
\int_{\RR^d}|n_0(t_2)-n_0(t_1)|\zeta_R \,dx \leq \eta.
$$
We conclude that $n_0 \in C((0,T),L^1(\RR^d))$.

\paragraph{Initial trace}
For any test function $0 \leq \zeta(x) \leq 1$, we have from \eqref{eq:nH},
\begin{align*}
 \int_{\RR^d} n_{\epsilon}(t) \zeta\,dx -\int_{\RR^d} n_{\epsilon}^{ini} \zeta\,dx   & = \int_{0}^t  \int_{\RR^d} (\Delta H(n_{\epsilon}) + n_{\epsilon} G(p_{\epsilon})) \zeta \,dxds \\
 & = \int_{0}^t  \int_{\RR^d}  (H(n_{\epsilon})  \Delta \zeta + n_{\epsilon} G(p_{\epsilon})\zeta ) \,dxds.
 \end{align*} 
 Letting $\epsilon$ going to $0$, we obtain with \eqref{hypini},
 \begin{align*}
 \int_{\RR^d} n_{0}(t) \zeta\,dx -\int_{\RR^d} n_{0}^{ini} \zeta\,dx   & = \int_{0}^t  \int_{\RR^d} (p_{0} \Delta \zeta +n_{0} G(p_{0})\zeta)\,dxds.
 \end{align*} 
 Letting $t \rightarrow 0$  we can conclude that $n_{0}(0) = n_{0}^{ini}$.

\section{Complementary relation}\label{sec:compl}

In this section we prove the complementary relation
$$
p_0^2 (\Delta p_0 + G(p_0)) =0.
$$
In the weak sense, this identity reads, for any test function $\phi$,
\begin{equation}\label{compldis}
\iint_{Q_T} \left(- 2\phi p_{0}  |\nabla p_0|^2 -  p_0^2 \nabla p_0\cdot \nabla \phi + \phi p_0^2 G(p_0)\right) \,dxdt = 0.
\end{equation}
The proof is divided into two steps.  \\

{\it 1st step.} 
In this first step we prove the inequality $\geq 0$ in \eqref{compldis}.
We start with the pressure equation \eqref{eq:p1} that we multiply by $ \epsilon$
$$ 
\epsilon \partial_t p_{\epsilon} - p_{\epsilon}(\epsilon+p_{\epsilon})\Delta p_{\epsilon} - \epsilon |\nabla p_{\epsilon}|^2= p_{\epsilon}(\epsilon+p_{\epsilon}) G(p_{\epsilon}).
$$
We multiply by a test function $\phi\in \mathcal{D}((0,T)\times \RR^d)$ and integrate,
\begin{align*}
 \iint_{Q_T} p_{\epsilon}^2 \phi (\Delta p_{\epsilon}+G(p_{\epsilon})) \,dxdt & =\epsilon \iint_{Q_T} \phi ( \partial_t p_{\epsilon}  - |\nabla p_{\epsilon}|^2 - p_{\epsilon}(\Delta p_{\epsilon}+G(p_{\epsilon}) )\,dxdt \\
 & = \epsilon \iint_{Q_T} \left( \phi\partial_t p_{\epsilon} +p_{\epsilon} \nabla p_{\epsilon}\cdot \nabla \phi - \phi p_{\epsilon} G(p_{\epsilon})\right) \,dxdt,
\end{align*}
where we use an integration by parts for the last identity.
From the estimates in Lemma \ref{lem:estim}, we have 
\begin{align*}
& \left| \epsilon \iint_{Q_T} \big(\phi\partial_t  p_{\epsilon} +p_{\epsilon} \nabla p_{\epsilon} \cdot\nabla \phi - \phi p_{\epsilon} G(p_{\epsilon})\big)\,dxdt \right| \\
& \qquad \leq  \epsilon \left(\|\phi \|_{L^\infty} \| \partial_t p_\epsilon \|_{L^1(Q_T)} + \|\nabla \phi\|_{L^\infty} P_M \| \nabla p_\epsilon \|_{L^1(Q_T)} +\|\phi\|_{L^\infty} G_{m} \| p_\epsilon \|_{L^1(Q_T)}\right) \\
& \qquad \underset{\epsilon \rightarrow 0}{\longrightarrow} 0.
\end{align*}
We deduce that for any test function $\phi\in \mathcal{D}((0,T)\times \RR^d)$,
\begin{equation}\label{step1}
\iint_{Q_T} \left(- 2\phi p_{\epsilon}  |\nabla p_{\epsilon}|^2 -  p_{\epsilon}^2 \nabla p_{\epsilon} \nabla \phi + \phi p_{\epsilon}^2 G(p_{\epsilon})\right) dxdt \underset{\epsilon \rightarrow 0}{\longrightarrow} 0.
\end{equation}
Since we have strong convergence of $(p_{\epsilon})_\epsilon$ and weak convergence of $(\nabla p_{\epsilon})_\epsilon$,
we can pass into the limit in the last two term in \eqref{step1},
$$ 
\iint_{Q_T} \left(- p_{\epsilon}^2 \nabla p_{\epsilon} \nabla \phi + \phi p_{\epsilon}^2 G(p_{\epsilon})\right)\,dxdt 
\underset{\epsilon \rightarrow 0}{\longrightarrow}
\iint_{Q_T} \left(- p_0^2 \nabla p_0 \nabla \phi + \phi p_0^2 G(p_0)\right) dxdt.
$$
Now we are looking for the limit of the first term in \eqref{step1}. We have
$ p_{\epsilon}  |\nabla p_{\epsilon}|^2  = \frac{4}{9} |\nabla p_{\epsilon}^{3/2}|^2$.
By weak convergence of $ \nabla p_{\epsilon}^{3/2} = p_{\epsilon}^{1/2}  \nabla p_{\epsilon}$ and with Jensen inequality
(since $x \mapsto x^2$ is convex),
$$ \underset{\epsilon\to 0}{\lim\inf} \iint_{Q_T} \phi p_{\epsilon}  |\nabla p_{\epsilon}|^2\,dxdt \leq \iint_{Q_T} \phi |\nabla p_0^{3/2}|^2\,dxdt.$$
Thus, we conclude from \eqref{step1} that
$$ 0 \leq \iint_{Q_T} \left(- 2\phi p_{0}  |\nabla p_0|^2 -  p_0^2 \nabla p_0 \nabla \phi + \phi p_0^2 G(p_0)\right) \,dxdt, $$
which is a first inequality for \eqref{compldis}. \\

{\it 2nd step.} Now we want to show the reverse inequality, i.e.
$$0 \geq \iint_{Q_T} \left(- 2\phi p_{0}  |\nabla p_0|^2 -  p_0^2 \nabla p_0 \nabla \phi + \phi p_0^2 G(p_0)\right) \,dxdt. $$
We know that
$$
\partial_t n_{\epsilon} -\Delta q_{\epsilon} = n_{\epsilon} G(p_{\epsilon}),
$$
with $ q_{\epsilon}  = p_{\epsilon} - \epsilon \ln(p_{\epsilon} +\epsilon).$
Thanks to the inequality $ \epsilon \ln(\epsilon) \leq \epsilon \ln(p_{\epsilon} +\epsilon)  \leq \epsilon \ln(p_M +\epsilon) $, and the strong convergence $p_\epsilon\to p_0$, we know that $q_{\epsilon} \rightarrow p_0$ as $\epsilon\to 0$.
Because
$$ \Delta q_{\epsilon}= \partial_t n_{\epsilon} -n_{\epsilon} G(p_{\epsilon}),$$
we deduce from Lemma \ref{lem:estim} that $ \Delta q_{\epsilon} \in L^{\infty}([0,T];L^{1}(\RR^d)) $. 
It gives us compactness in space but not in time. Thus, following the idea of \cite{PQV}, 
we use a regularization process '\`a la Steklov'. 

Let introduce a time regularizing kernel $\omega_{\eta} \geq 0$ such that $\mbox{supp}(\omega_{\eta}) \subset \RR_-$.
Then with the notations $ n_{\epsilon,\eta} = \omega_{\eta} *_t n_{\epsilon}$, $q_{\epsilon,\eta}= \omega_{\eta} *_t q_{\epsilon}$,
where the convolution holds only in the time variable,
\begin{equation}\label{eqconv}
 \partial_t n_{\epsilon,\eta} -\Delta q_{\epsilon,\eta} = (n_{\epsilon} G(p_{\epsilon}))*\omega_{\eta}
\end{equation}
We denote
$ U_{\epsilon} =\Delta q_{\epsilon,\eta} $, then
\begin{align*}
 U_{\epsilon} &= \partial_t n_{\epsilon,\eta}-(n_{\epsilon} G(p_{\epsilon}))*\omega_{\eta} \\
 &= n_{\epsilon} *\partial_t \omega_{\eta}-(n_{\epsilon} G(p_{\epsilon}))*\omega_{\eta}
 \end{align*}
 Since $ n_{\epsilon}$ and $n_{\epsilon} G(p_{\epsilon})$ are uniformly bounded in $W^{1,1}(Q_T)$ from Lemma \ref{lem:estim}, 
$(U_{\epsilon})_\epsilon $ is bounded in $W^{1,1}(Q_T)$ and we can extract a converging subsequence, still denoted $(U_{\epsilon})_\epsilon$, converging towards $U_0$ in $L^1_{loc}(\RR^d)$ for $\eta$ fixed. Moreover
$$ U_0 = \Delta p_{0}*\omega_{\eta}.$$
 
We multiply \eqref{eqconv} by $P'(n_{\epsilon})= \frac{\epsilon}{(1-n_{\epsilon})^2}= \frac{1}{\epsilon} (p_{\epsilon}+\epsilon)^2$,
$$
\frac{\epsilon}{(1-n_{\epsilon})^2} \partial_t n_{\epsilon,\eta} -\frac{1}{\epsilon} (p_{\epsilon}+\epsilon)^2\Delta q_{\epsilon,\eta} =\frac{1}{\epsilon} (p_{\epsilon}+\epsilon)^2 (n_{\epsilon} G(p_{\epsilon}))*\omega_{\eta}.
$$
Then, passing to the limit $\epsilon\to 0$, we obtain, thanks to the above remark
$$ 
\frac{\epsilon^2}{(1-n_{\epsilon})^2} \partial_t n_{\epsilon,\eta} \underset{\epsilon\to 0}{\longrightarrow} 
p_0^2 \Delta p_{0}*\omega_{\eta} + p_0^2(n_{0} G(p_{0}))*\omega_{\eta}.
$$
So we are left to prove that for any $\eta >0$, we have
$$ \lim_{\epsilon\to 0} \frac{\epsilon^2}{(1-n_{\epsilon})^2} \partial_t n_{\epsilon,\eta} \leq 0.$$

 We compute for a fixed $\eta >0$,
\begin{align*}
 \frac{\epsilon^2}{(1-n_{\epsilon})^2} \partial_t n_{\epsilon,\eta}(t,x)  & =
\int_{\RR}  \frac{\epsilon^2}{(1-n_{\epsilon}(t,x))^2} \partial_t n_{\epsilon} (s,x)\omega_{\eta}(t-s,x) ds  \\
& = \int_{\RR}  \frac{\epsilon^2}{(1-n_{\epsilon}(s,x))^2}  \partial_t n_{\epsilon} (s,x)\omega_{\eta}(t-s,x) ds \\
&+\int_{\RR} ( \frac{\epsilon^2}{(1-n_{\epsilon}(t,x))^2}- \frac{\epsilon^2}{(1-n_{\epsilon}(s,x))^2}) (\partial_t n_{\epsilon} (s,x)+\frac{C}{s})\omega_{\eta}(t-s,x) ds \\
&-C \int_{\RR} ( \frac{\epsilon^2}{(1-n_{\epsilon}(t,x))^2}- \frac{\epsilon^2}{(1-n_{\epsilon}(s,x))^2}) \frac{\omega_{\eta}(t-s,x)}{s} ds \\
& = \text{ I }_{\epsilon}+\text{ II }_{\epsilon}+\text{ III }_{\epsilon},
 \end{align*}
 where $C$ is a constant such that $ \partial_t n_{\epsilon} (s,x)+\frac{C}{t} \geq 0$. 
 
 For the first term we have
 \begin{align*}
\int_{\RR^d} |\text{ I }_{\epsilon}| dxds &
\leq  \epsilon \iint_{Q_T} |\partial_t p_{\epsilon} (s,x)|\omega_{\eta}(t-s,x) dxds \\
&\leq  \epsilon \|\omega_{\eta}\|_{L^\infty} \| \partial_t p_{\epsilon}\|_{L_1(Q_T)} \leq  \epsilon C_{\eta}\\
& \underset{\epsilon\to 0}{\longrightarrow} 0.
\end{align*}

For the second term, we have
$$ \frac{\epsilon^2}{(1-n_{\epsilon}(t,x))^2}  = (p_{\epsilon}+\epsilon)^2$$ and $\partial_t (p_{\epsilon}+\epsilon)^2 = 2(p_{\epsilon}+\epsilon) \partial_t p_{\epsilon} \geq - \frac{C'}{t} $. Let $0 \leq \xi \in \mathcal{C}^{\infty}_c(Q)$ and $\tau>0$ the smallest time in its support, we then have for $t \geq \tau$
$$ \partial_t (p_{\epsilon}+\epsilon)^2(t,x) \geq - \frac{C'}{\tau} .$$ 
So integrating on $(t,s) \subset (\tau, + \infty)$
$$ \frac{\epsilon^2}{(1-n_{\epsilon}(t,x))^2}- \frac{\epsilon^2}{(1-n_{\epsilon}(s,x))^2} \leq  \frac{C'}{\tau} (s-t) .$$
Then
$$ \iint_{Q} \xi \text{ II }_{\epsilon} \leq \frac{C'}{\tau} \eta  \iint_{Q} \int_{\RR} (\partial_t n_{\epsilon} (s,x)+\frac{C}{\tau})\omega_{\eta}(t-s,x) ds dx dt \leq  C'_{\tau } \eta,$$
where we use the bound on $\partial_t n$ in Lemma \ref{estimdtn}.

For the third term, since $s \geq t >0$, for any test function $\xi$ as above,
 \begin{align*}
 \iint_{Q} \xi \text{ III }_{\epsilon} = & -C \iint_{Q}  \xi\int_{\RR} ( (p_{\epsilon}(t)+\epsilon)^2- (p_{\epsilon}(s)+\epsilon)^2) \frac{\omega_{\eta}(t-s,x)}{s} ds  \\
 & \underset{\epsilon\to 0}{\longrightarrow} -C \iint_{Q} \xi \int_{\RR} ( p_{0}(t)^2- p_{0}(s)^2) \frac{\omega_{\eta}(t-s,x)}{s} dsdx dt\\
  & \underset{\epsilon\to 0}{\longrightarrow} -C \iint_{Q}  \xi \left[p_0^2(t) \int_{\RR} \frac{\omega_{\eta}(t-s,x)}{s} ds - \int_{\RR} \frac{p_{0}(s)^2}{s} \omega_{\eta}(t-s,x)ds\right] dx dt \\
& \qquad =  \underset{\eta\to 0}{o}(1).
 \end{align*}
 So for all test function $\xi$ as above, and all $\eta>0$,
$$ \iint_{Q_T} \xi (p_{0}^2  \Delta p_{0}*\omega_{\eta} + p_0^2 (n_0G(p_0))*\omega_{\eta}) dx dt \leq  \underset{\eta\to 0}{o}(1). $$
Now it remain to pass to the limit $ \eta \rightarrow 0$ in the regularization process. 
Thanks to an integration by parts,
$$
0 \geq \iint_{Q_T} (- 2\xi p_{0} \nabla p_0\cdot\nabla p_0*\omega_\eta -  p_0^2 \nabla \xi\cdot\nabla p_0 *\omega_\eta
+ \xi p_0^2 (n_0 G(p_0))*\omega_\eta )dxdt. $$
From the $L^2$ estimate on $\nabla p_0$ (Lemma \ref{lem:L2dp}) and the $L^1\cap L^\infty$ estimate on $p_0$ (Lemma \ref{lem:estim}),
we deduce that we can pass to the limit $\eta\to 0$ and get
$$
0 \geq \iint_{Q_T} (- 2\xi p_{0} |\nabla p_0|^2 -  p_0^2 \nabla \xi\cdot\nabla p_0 + \xi p_0^2 n_0 G(p_0))dxdt. 
$$

 Finally, from \eqref{n0p0}, we have $p_0 n_0 = p_0$. It concludes the proof.

%
%


\end{document}